\documentclass[10pt, reqno]{amsart}
\usepackage{graphicx, amssymb, amsmath, amsthm,comment}
\numberwithin{equation}{section}
\usepackage{color}

\usepackage{tikz}

\usepackage[backref=page]{hyperref}  
\usepackage{amscd}

\newcommand{\R}{\mathbb{R}}

\newtheorem{theorem}{Theorem}[section]

\newtheorem{lemma}[theorem]{Lemma}

\theoremstyle{definition}

\newtheorem{remark}[theorem]{Remark}

\makeatletter
\newcommand{\Extend}[5]{\ext@arrow0099{\arrowfill@#1#2#3}{#4}{#5}}
\makeatother

\begin{document}
\title[Decoupling]{$\ell^2$ decoupling for certain surfaces of finite type in $\mathbb{R}^3$}

\author{Zhuoran Li}
\address{The Graduate School of China Academy of Engineering Physics, P. O. Box 2101,\ Beijing,\ China,\ 100088 ;}
\email{lizhuoran18@gscaep.ac.cn}

\author{Jiqiang Zheng}
\address{Institute of Applied Physics and Computational Mathematics, Beijing 100088}
\email{zhengjiqiang@gmail.com}

\begin{abstract}
In this article, we establish an $\ell^2$ decoupling inequality for the surface
\[F_4^2:=\Big\{(\xi_1,\xi_2,\xi_1^4+\xi_2^4): (\xi_1,\xi_2) \in [0,1]^2\Big\}\]
associated with the decomposition adapted to finite type geometry from our previous work \cite{LMZ}.
The key ingredients of the proof include the so-called generalized rescaling technique, an $\ell^2$ decoupling inequality for the surfaces
\[\Big\{(\xi_1,\xi_2,\phi_1(\xi_1)+\xi_2^4): (\xi_1,\xi_2) \in [0,1]^2\Big\}\]
with $\phi_1$ being non-degenerate,
reduction of dimension arguments and induction on scales.
\end{abstract}

\maketitle

\begin{center}
 \begin{minipage}{100mm}
 \end{minipage}
 \end{center}





\section{Introduction and main result}\label{sec:intmai}

\noindent

Decoupling inequalities was introduced by Wolff \cite{Wolff00} in connection with the local smoothing estimates for the solution to the wave equation. Bourgain made important progress in developing the decoupling theory for hypersurfaces in \cite{Bo13} via the multilinear restriction theory of Bennett, Carbery and Tao \cite{BCT}, and applied  decoupling inequalities to the study of exponential sums. In 2015, Bourgain and Demeter \cite{BD15} proved the sharp $\ell^2$ decoupling inequality for compact hypersurfaces with positive definite second fundamental form,
which has a wide range of important consequences. One of them is the validity of the discrete restriction conjecture, which implies the full range of expected $L^p_{t,x}$ Strichartz estimates for the Schr\"{o}dinger equation on both the rational and irrational torus. The second striking application is an improvement on the local smoothing conjecture for the wave equation. In particular, inspired by Bourgain-Demeter's work \cite{BD15}, Guth, Wang and Zhang \cite{GWZ20} established a sharp reverse square function estimate for the cone in $\mathbb{R}^3$, and solved  the local smoothing conjecture for the wave equation in $\mathbb{R}^{2+1}$. Thirdly, by establishing the sharp decoupling inequality for the moment curve, Bourgain, Demeter and Guth \cite{BDG16} proved the main conjecture in Vinogradov's mean value theorem for degrees higher than three, which is a remarkable application of decoupling methods in number theory.

Recently, the decoupling theory for non-degenerate hypersurfaces has been well understood. However, the decoupling theory for degenerate hypersurfaces such as hypersurfaces of finite type is interesting but still unknown. 
In this paper, we will study the $\ell^2$ decoupling problem for certain  surfaces of finite type in $\mathbb{R}^3$. More precisely, we consider the surface in $\mathbb{R}^3$ given by
\begin{equation*}
  F^2_m:=\Big\{(\xi_1,\xi_2,\xi_1^m+\xi_2^m): (\xi_1,\xi_2) \in [0,1]^2\Big\},\;m\geq2.
\end{equation*}
For any function $g\in L^1([0,1]^2)$ and each subset $Q\subset [0,1]^{2}$, we denote the corresponding Fourier extension operator by
\begin{equation*}
  \mathcal E_{Q}g(x):= \int_Q g(\xi_1,\xi_2)e\big(x_1\xi_1+x_2\xi_2+x_3(\xi^m_1+\xi_2^m)\big)\;d\xi_1\;d\xi_2,
\end{equation*}
where $e(t)=e^{2\pi i t}$ for $t \in \mathbb{R}$, and $x=(x_1,x_2,x_3)\in\R^3.$

For $m=2$, the surface $F^2_2$ is exactly a paraboloid over the region $[0,1]^2$.
By making use of Bourgain-Guth methods in \cite{BG}, parabolic rescaling and induction on scales,
Bourgain and Demeter \cite{BD15} have established the sharp $\ell^2$ decoupling inequality for the paraboloid in $\mathbb{R}^n$ as follows.

\begin{theorem}\label{thm:BD}
For $2\leq p\leq \frac{2(n+1)}{n-1}$ and all $\varepsilon > 0$, there exists a constant $C(\varepsilon,p)$ such that
\begin{equation}\label{equ:bd} \|\mathcal {E}^{Par}_{[0,1]^{n-1}}g\|_{L^{p}(B^n_R)}\leq C(\varepsilon,p)R^{\varepsilon}\Big(\sum_{\theta:\; R^{-1/2}-\text{cubes\; in\;} [0,1]^{n-1}}\|\mathcal {E}^{Par}_{\theta}g\|^2_{L^{p}(w_{B^n_R})}\Big)^{1/2},
\end{equation}
where $\mathcal{E}^{Par}_{[0,1]^{n-1}}$ denotes the Fourier extension operator associated with the paraboloid $P^{n-1}$ and $w_{B^n_R}(x)= \big(1+\frac{\vert x-x_0 \vert}{R}\big)^{-100n}$ denotes the standard weight function adapted to the ball $B^n_R$ in $\mathbb{R}^n$ centered at $x_0$ with radius $R$.
\end{theorem}
We will also use the following equivalent version of \eqref{equ:bd}
\begin{equation}\label{equ:bd1}
  \|F\|_{L^{p}(\mathbb{R}^n)}\leq C(\varepsilon,p)R^{\varepsilon}\Big(\sum_{\bar{\theta}}\|F_{\bar{\theta}}\|^2_{L^{p}(\mathbb{R}^n)}\Big)^{1/2},\;\; 2\leq p \leq \frac{2(n+1)}{n-1},
\end{equation}
where ${\rm supp}\;\hat{F} \subset \mathcal{N}_{\frac{1}{R}}(P^{n-1})$, the $\bar{\theta}$'s are finitely overlapping $R^{-1/2}$-slabs in $\mathbb{R}^n$ which cover $\mathcal{N}_{\frac{1}{R}}(P^{n-1})$ (the $\tfrac1R-$neighborhood of $P^{n-1}$)  and $F_{\bar{\theta}}:= \mathcal{F}^{-1}\big(\hat{F}\chi_{\bar{\theta}}\big)$.
For the equivalence of \eqref{equ:bd} and \eqref{equ:bd1}, we refer to \cite{BD17}.

For $m>2$, the Gaussian curvature of the surface $F_m^2$ vanishes when $\xi_1=0$ or $\xi_2=0$. From now on, we focus on the case $m=4$, since the cases for even numbers larger than four can be deduced by the same argument. It is surprised for us that such a degenerate surface might possess the same range of exponent as in the $\ell^2$ decoupling inequality for the paraboloid. For this purpose, we adopt the notations as in the previous work \cite{LMZ}. At first, given $R\gg1$, we divide $[0,1]$ into
\begin{equation*}
  [0,1]=\bigcup_k I_k ,
\end{equation*}
where $I_0=[0,R^{-\frac14}]$ and
$$I_k=[2^{k-1}R^{-\frac14},2^{k}R^{-\frac14}], \;\;\text{for}\;\; 1\leq k \leq \big[\tfrac{1}{4}\log_{2}R\big].$$
For each $k\geq 1$, we divide $I_k$ further into
\[I_k=\bigcup_{\mu=1}^{2^{2(k-1)}} I_{k,\mu},\]
with
$$I_{k,\mu}=\big[2^{k-1}{R^{-\frac14}}+(\mu-1)2^{-(k-1)}R^{-\frac14},2^{k-1}R^{-\frac14}+ \mu 2^{-(k-1)}R^{-\frac14}\big].$$
Thus, we have the following decomposition
\begin{equation}\label{equ:r4regdec}
  [0,1]^{2}=\bigcup_{\theta\in \mathcal{F}_3(R,4)} \theta,
\end{equation}
where
\begin{align*}
   \mathcal{F}_3(R,4):=\big\{&I_{k_1,\mu_1}\times I_{k_2,\mu_2},\; I_0\times I_{k_2,\mu_2},\; I_{k_1,\mu_1}\times I_0,\; I_0\times I_0:\\
   &\quad 1\leq k_j \leq \big[\tfrac{1}{4}\log_{2}R\big],\;1\leq \mu_j\leq 2^{2(k-1)},\;j=1,2\big\}.
\end{align*}
We refer to \cite{LMZ} for the detailed decomposition \eqref{equ:r4regdec}. Buschenhenke \cite{Bus} utilized the analogous decomposition to study the restriction estimates for certain conic surfaces of finite type.

Our main result is the following decoupling inequality associated with the surface $F^2_4$ based on the decomposition \eqref{equ:r4regdec}.

\begin{theorem}\label{thm:main}
For $2\leq p\leq 4$ and any $\varepsilon > 0$, there exists a constant $C(\varepsilon,p)$ such that
\begin{equation}\label{equ:goalm}
  \|\mathcal E_{[0,1]^2}g\|_{L^{p}(B_R)}\leq C(\varepsilon,p)R^{\varepsilon}\Big(\sum_{\theta\in \mathcal{F}_3(R,4)}\|\mathcal E_{\theta}g\|^2_{L^{p}(w_{B_R})}\Big)^{1/2},
\end{equation}
with $w_{B_R}(x)$ being  the standard weight function  as in Theorem \ref{thm:BD}.
\end{theorem}

\begin{remark}\label{rem:Xi}

Biswas, Gilula, Li, Schwend and Xi \cite{Xi} studied the $\ell^2$ decoupling problem for curves of finite type in the plane. Their decoupling inequality is based on a traditional decomposition as in the case of the parabola, which seems ineffectively in higher dimensional cases because of failure of standard parabolic rescaling. Recently, Yang \cite{Yang21} proved an uniform $\ell^2$ decoupling inequality for polynomials in $\R^2$. We refer the readers to Section 2.2 in \cite{LMZ} for the details on the decomposition adapted to finite type surfaces.

\end{remark}

\vskip 0.2in

The paper is organized as follows. In Section 2, we give the outline of the proof of Theorem \ref{thm:main}. In Section 3, we establish an $\ell^2$ decoupling inequality for the surfaces
\[\Big\{(\xi_1,\xi_2,\phi_1(\xi_1)+\xi_2^4): (\xi_1,\xi_2) \in [0,1]^2\Big\}\]
with $\phi_1$ being non-degenerate as a key lemma. With this in hand, we utilize Bourgain-Demeter's decoupling inequality for the perturbed paraboloid, reduction of dimension arguments and induction on scales to prove Theorem \ref{thm:main}.

\vskip 0.2in

{\bf Notations:} For nonnegative quantities $X$ and $Y$, we will write $X\lesssim Y$ to denote the estimate $X\leq C Y$ for some $C>0$. If $X\lesssim Y\lesssim X$, we simply write $X\sim Y$. Dependence of implicit constants on the power $p$ or the dimension will be suppressed; dependence on additional parameters will be indicated by subscripts. For example, $X\lesssim_u Y$ indicates $X\leq CY$ for some $C=C(u)$. We denote $e(t)=e^{2\pi i t}$. We denote $[x]$ to be the greatest integer not large than $x$. We use $B(x_0,r)$ to denote an arbitrary ball centred at $x_0$ with radius $r$ in $\mathbb{R}^3$ and abbreviate it by $B_r$ in the context. For any region $\Omega \subset \mathbb{R}^n$, we denote the characteristic function on $\Omega$ by $\chi_{\Omega}$. In $\R^n$, we denote $R^{-1/2}\times \cdot\cdot\cdot \times R^{-1/2} \times R^{-1}$-rectangle to be an $R^{-1/2}$-slab in $\mathbb{R}^n$. Define the Fourier transform on $\mathbb{R}^n$ by
\begin{equation*}
\aligned \widehat{f}(\xi):= \int_{\mathbb{R}^n}e^{- 2\pi ix\cdot \xi}f(x)\,dx,
\endaligned
\end{equation*}
and the inverse Fourier transform by
\begin{equation*}
\aligned \mathcal{F}^{-1}{f}(x):= \int_{\mathbb{R}^n}e^{2\pi ix\cdot \xi}f(\xi)\,d\xi.
\endaligned
\end{equation*}




\section{Outline of the proof of Theorem \ref{thm:main}}\label{sec:2}

Now, we give the sketch of the proof of Theorem \ref{thm:main}. By a direct computation, the Gaussian curvature of the surface $F^2_4$ has
\begin{equation}\label{equ:gausscur}
  k=\frac{144\xi^2_1\xi^2_2}{(1+16\xi^6_1+16\xi^6_2)^2}.
\end{equation}
We observe that the surface $F^2_4$ has positive definite second fundamental form if both $\xi_1$ and $\xi_2$ are away from zero. In this region, one can adopt Bourgain-Demeter's decoupling inequality for the perturbed paraboloid. However, when $\xi_1=0$ or $\xi_2=0$, the Gaussian curvature vanishes. Our strategy is to divide $[0,1]^2$ into $[0,1]^2=\bigcup\limits_{j=0}^3\Omega_j$ as in Fig.1,
\begin{center}
 \begin{tikzpicture}[scale=0.6]

\draw[->] (-0.2,0) -- (5,0) node[anchor=north]   {$\xi_1$};
\draw[->] (0,-0.2) -- (0,5)  node[anchor=east] {$\xi_2$};

 \draw[thick] (0,4) -- (4,4)
 (4,0)--(4,4)
 (0,1)--(4,1)
 (1,0)--(1,4) ;

 \draw (-0.15,0) node[anchor=north] {O}
      (1.35,-0.2) node[anchor=north] {$K^{-\frac14}$}
     (4,-0.08) node[anchor=north] {$1$}
    (2.5,2.5) node[anchor=north] {$\Omega_0$}
     (2.5,0.8) node[anchor=north] {$\Omega_1$}
     (0.5,2.5) node[anchor=north] {$\Omega_2$}
      (0.5,0.8) node[anchor=north] {$\Omega_3$};

 \draw  (-0.35, 1) node[anchor=east] {$K^{-\frac14}$}
       (0,4) node[anchor=east] {$1$};


\path (2.5,-2.0) node(caption){Fig. 1};
 \end{tikzpicture}
\end{center}
where
\begin{align*}
&\Omega_0:=[K^{-\frac14},1]\times[K^{-\frac14},1],\quad
\Omega_1:=[K^{-\frac14},1]\times[0,K^{-\frac14}],\\
&\Omega_2:=[0,K^{-\frac14}]\times[K^{-\frac14},1],\quad
\Omega_3:=[0,K^{-\frac14}]\times[0,K^{-\frac14}].
\end{align*}
In the region $\Omega_3$, we will use the rescaling technique. While for the regions $\Omega_1$ and $\Omega_2$, we reduce them to a lower dimensional decoupling problem.

For technical reasons, $K^{-1/4}$, $R^{-1/4}$ and $\big(\tfrac{R}{K}\big)^{-1/4}$ should be dyadic numbers satisfying $1\ll K \ll R^{\varepsilon}$ for any fixed $\varepsilon > 0$. Therefore, we choose $K=2^{4s}$  and $R=2^{4l}\; (s,l \in \mathbb{N})$ to be large numbers satisfying $K \approx \log
R$.
By the Minkowski inequality and Cauchy-Schwartz inequality, we have
 \begin{align}\label{equ:e01intr}
    \|\mathcal E_{[0,1]^2}g\|_{L^p(B_R)}\leq & 2\Big(\sum_{j=0}^{3} \|\mathcal E_{\Omega_j}g\|^2_{L^p(B_R)}\Big)^{1/2}.
 \end{align}
Let $\mathcal{D}_p(R)$ denote the least number such that
\begin{equation}\label{equ:defqpr1}
  \|\mathcal E_{[0,1]^2}g\|_{L^p(B_R)}\leq \mathcal{D}_p(R)\Big(\sum_{\theta\in \mathcal{F}_3(R,4)}\|\mathcal E_{\theta}g\|^2_{L^{p}(w_{B_R})}\Big)^{1/2}.
\end{equation}
{\bf Case 1: The contribution comes from $\Omega_0$.} Roughly speaking, the strategy is to transform each small piece to a non-degenerate one. Then we employ Bourgain-Demeter's decoupling inequality for each non-degenerate piece. We first divide $\Omega_0$ into $\Omega_0= \bigcup \Omega_{\lambda,\sigma}$ with
 \begin{equation}\label{omega0decomposition}
\Omega_{\lambda,\sigma}=\big\{(\xi_1,\xi_2):\; \lambda \leq \xi_1 \leq 2\lambda,\;\;\sigma \leq \xi_2 \leq 2\sigma\big\},
\end{equation}
and
$\lambda,\sigma \in [K^{-\frac14}, \frac{1}{2}]$ are both dyadic numbers.
We know that the Gaussian curvature of $F^2_4$ is essentially a constant in the region  $\Omega_{\lambda,\sigma}.$
It is reasonable to divide $\Omega_{\lambda,\sigma}$ into $\lambda^2 K^{\frac12}\times\sigma^2 K^{\frac12}$ pieces of dimension $\lambda^{-1}K^{-\frac12}\times\sigma^{-1}K^{-\frac12}$ equally
\begin{align}\nonumber
   \Omega_{\lambda,\sigma}=&\bigcup_{1\le j\le\lambda^2K^{\frac12}\atop 1\le m\le \sigma^2K^{\frac12}}
\Big[\lambda+\frac{j-1}{\lambda K^{1/2}}, \lambda+\frac{j}{\lambda K^{1/2}} \Big]\times\Big[\sigma+\frac{m-1}{\sigma K^{1/2}}, \sigma+\frac{m}{\sigma K^{1/2}}\Big]\\\label{equ:taudef}
=:& \bigcup \tau_{\lambda,\sigma}^{j,m}.
\end{align}
For simplicity, we abbreviate $\tau_{\lambda,\sigma}^{j,m}$ by $\tau.$

We can implement the generalized rescaling to each $\tau$. Without loss of generality, we may assume that
\begin{equation}\label{equ:tauspec}
  \tau=[\lambda,\lambda+\lambda^{-1} K^{-\frac12}]\times [\sigma,\sigma+\sigma^{-1}K^{-\frac12}].
\end{equation} We take the change of variables
\begin{equation}\label{caseachangeofvariables}
  \xi_1=\lambda+\frac{\eta_1}{\lambda K^{\frac12}},\;
  \xi_2=\sigma+\frac{\eta_2}{\sigma K^{\frac12}}
  \end{equation}
to get
\begin{align}\nonumber
  \vert \mathcal E_{\tau}g(x)\vert=&\Big| \int_{[0,1]^2}\tilde{g}(\eta_1,\eta_2)e[\tilde{x}_1\eta_1+\tilde{x}_2\eta_2+\tilde{x}_3\psi_0(\eta_1,\eta_2)]
  \;d\eta_1\;d\eta_2\Big|\\\label{caseaoperator}
  =:&\vert {\mathcal E}^{\rm Parp}_{[0,1]^2}\tilde{g}(\tilde{x})\vert,
\end{align}
where
\begin{equation}\label{outlineomega0taurescaling}\left\{
  \begin{aligned}
&\tilde{x}=\big(\lambda^{-1}K^{-\frac12}x_1+4\lambda^2K^{-\frac12}x_3,\sigma^{-1}K^{-\frac12}x_2+4\sigma^2K^{-\frac12}x_3,K^{-1}x_3\big),\\
&\tilde{g}(\eta_1,\eta_2)=\lambda^{-1}\sigma^{-1}K^{-1}g\Big(\lambda+\frac{\eta_1}{\lambda K^{\frac12}},\sigma+\frac{\eta_2}{\sigma K^{\frac12}}\Big),\\
&\psi_0(\eta_1,\eta_2)=(6\eta_1^2+4\lambda^{-2}K^{-\frac12}\eta_1^3+\lambda^{-4}K^{-1}\eta_1^4)
+(6\eta_2^2+4\sigma^{-2}K^{-\frac12}\eta_2^3+\sigma^{-4}K^{-1}\eta_2^4).
\end{aligned}\right.
\end{equation}
Here $\mathcal{E}^{\rm Parp}_{[0,1]^2}$ denotes the Fourier extension operator associated with the phase function $\psi_0$. Since $\lambda,\sigma \geq K^{-\frac14}$ and $0\leq \xi_i \leq 1(i=1,2)$, it is easy to see that $\psi_0$ is of the form $\psi_0(\eta_1,\eta_2)= \phi_1(\eta_1)+\phi_2(\eta_2)$ satisfying $$\phi''_j \sim 1,\; \vert \phi^{(3)}_j \vert \lesssim 1,\; \vert \phi^{(4)}_j \vert \lesssim 1 \quad \text{and} \quad \phi^{(\ell)}_j =0,\; \ell \geq 5 \;  \text{on} \;[0,1],\; j=1,\;2.$$
Thus, we can adapt the argument in Section 7 of \cite{BD15} to obtain the following uniform decoupling inequality for $\Omega_0$:
\begin{equation}\label{equ:bd15}
  \|\mathcal E_{\Omega_0}g\|_{L^p(B_R)}\leq C(\varepsilon)K^{O(1)}R^{\varepsilon}\Big(\sum_{\theta \subset \Omega_0}\|\mathcal E_{\theta}g\|^2_{L^{p}(w_{B_R})}\Big)^{1/2},  \;\; 2\leq p\leq 4,
\end{equation}
where $K^{O(1)}$ denotes a fixed power of $K$, and $\theta \in \mathcal{F}_3(R,4)$. This will be shown in Subsection \ref{subs:3.1} below.

{\bf Case 2: The contribution comes from $\Omega_3=[0,K^{-\frac14}]\times[0,K^{-\frac14}]$.} To do it, we will apply rescaling to
$|\mathcal E_{\Omega_3}g|$ directly and transform $\|\mathcal E_{\Omega_3}g\|_{L^p(B_R)}$ to $\|\mathcal{E}_{[0,1]^2}\tilde{g}\|_{L^p(B_{R/K})}$. More precisely, we take the change of variables
\begin{equation}\label{casebchangeofvariables}
  \xi_1=K^{-\frac14}\eta_1,\;
  \xi_2=K^{-\frac14}\eta_2,
\end{equation}
and obtain
\begin{align}\nonumber
  \vert \mathcal E_{\tau}g(x)\vert=&\Big| \int_{[0,1]^2}\tilde{g}(\eta_1,\eta_2)e[\tilde{x}_1\eta_1+\tilde{x}_2\eta_2+\tilde{x}_3\psi(\eta_1,\eta_2)]
  \;d\eta_1\;d\eta_2\Big|\\\label{caseboperator}
  =&\vert \mathcal E_{[0,1]^2}\tilde{g}(\tilde{x})\vert,
\end{align}
where
\begin{equation*}
  \left\{\begin{aligned}
&\tilde{x}=\big(K^{-\frac14}x_1,K^{-\frac14}x_2,K^{-1}x_3\big),\\
&\tilde{g}(\eta_1,\eta_2)=K^{-\frac12}g\big(K^{-\frac14}\eta_1,K^{-\frac14}\eta_2\big),\\
&\psi(\eta_1,\eta_2)=\eta_1^4+\eta_2^4.
\end{aligned}\right.
\end{equation*}

For the contribution of $\Omega_3$-part, we get by rescaling and the definition of $\mathcal{D}_p(\cdot)$
\begin{equation}\label{equ:omeest1}
  \|\mathcal E_{\Omega_3}g\|_{L^p(B_R)}\leq \mathcal{D}_p\Big(\frac{R}{K}\Big)\Big(\sum_{\theta\subset \Omega_3}\|\mathcal E_{\theta}g\|^2_{L^{p}(w_{B_R})}\Big)^{1/2},
\end{equation}
where $\theta \in \mathcal{F}_3(R,4)$.
We refer the details in Lemma \ref{lem:omega3} below.

{\bf Case 3: The contribution comes from $\Omega_1$-part and $\Omega_2$-part.} By symmetry, it suffices to treat $\Omega_1$-part.
To do it, we will establish an auxiliary lemma (see Lemma \ref{lem:24} below) to handle the Fourier extension operator associated with certain phase functions which are degenerate in $\xi_2$-variable but non-degenerate in $\xi_1$-variable. More precisely, we first divide $\Omega_1$ into $\Omega_1= \bigcup \Omega_{1,\lambda}$ with
 $$\Omega_{1,\lambda}=[\lambda,2\lambda]\times[0,K^{-\frac14}],$$
 where $\lambda\in\big[K^{-\frac14},\tfrac12\big]$ is a dyadic number. We divide $\Omega_{1,\lambda}$ further into
 \begin{align}\nonumber
   \Omega_{1,\lambda}= & \bigcup_{1\le j\le\lambda^2K^{\frac12}} \Big[\lambda+\frac{j-1}{\lambda K^{1/2}},\lambda+\frac{j}{\lambda K^{1/2}}\Big]\times [0,K^{-\frac14}]\\\label{equ:om1lamtau}
   =:&\bigcup \tau_{\lambda}^j.
 \end{align}
 We abbreviate $\tau_{\lambda}^{j}$ by $\tau.$ By reduction of dimension arguments and \eqref{equ:bd1} with $n=2$, we can derive that
 \begin{equation}\label{equ:omegall-1}
\| \mathcal{E}_{\Omega_1}g \|_{L^p(B_R)}\leq C_{\varepsilon}K^{\varepsilon}\Big(\sum_{\tau \subset \Omega_1} \| \mathcal{E}_{\tau}g \|^2_{L^p(w_{B_R})}\Big)^{1/2}.
\end{equation}
This will be shown in Lemma \ref{lem:3.3} below.

Next, we estimate the
$\| \mathcal{E}_{\tau}g \|_{L^p(B_R)}$ for the typical case $$\tau =[\lambda,\lambda+\lambda^{-1}K^{-\frac12}]\times[0,K^{-\frac14}].$$ The other cases can be deduced by the same argument.
We take the change of variables
\begin{equation}\label{casecchangeofvariables}
  \xi_1=\lambda+\frac{\eta_1}{\lambda K^{\frac12}},\;
  \xi_2=\frac{\eta_2}{K^{\frac14}}
\end{equation}
to obtain
\begin{align}\nonumber
  \vert \mathcal E_{\tau}g(x)\vert=&\Big| \int_{[0,1]^2}\tilde{g}(\eta_1,\eta_2)e[\tilde{x}_1\eta_1+\tilde{x}_2\eta_2+\tilde{x}_3\psi_1(\eta_1,\eta_2)]
  \;d\eta_1\;d\eta_2\Big|\\\label{casecoperator}
  =:&\vert \tilde{\mathcal{E}}_{[0,1]^2}\tilde{g}(\tilde{x})\vert,
\end{align}
where
\begin{equation}\label{equ:gtildef}\left\{\begin{aligned}
&\tilde{x}=\big(\lambda^{-1}K^{-\frac12}x_1+4\lambda^2K^{-\frac12}x_3,K^{-\frac14}x_2,K^{-1}x_3\big),\\
&\tilde{g}(\eta_1,\eta_2)=\lambda^{-1}K^{-3/4}g\big(\lambda+\frac{\eta_1}{\lambda K^{\frac12}},\frac{\eta_2}{K^{\frac14}}\big),\\
&\psi_1(\eta_1,\eta_2)
=(6\eta_1^2+4\lambda^{-2}K^{-\frac12}\eta_1^3+\lambda^{-4}K^{-1}\eta_1^4)
+\eta_2^4.
\end{aligned}\right.\end{equation}
Here $\tilde{\mathcal E}_{[0,1]^2}$ denotes the Fourier extension operator associated with phase functions of the form $\psi_1(\eta_1,\eta_2)=\phi_1(\eta_1)+\eta_2^4$ satisfying
$$\phi''_1 \sim 1,\; \vert \phi^{(3)}_1 \vert \lesssim 1,\; \vert \phi^{(4)}_1 \vert \lesssim 1 \;\text{and}\; \phi^{(\ell)}_1 =0,\; \ell \geq 5 \;\text{on}\; [0,1].$$
To deal with the $|\tilde{\mathcal E}_{[0,1]^2}\tilde{g}|$, we divide $[0,1]^2$ into $[0,1]^2= D_0 \bigcup D_1,$
where
$$D_0:= [0,1]\times [K^{-1/4},1],\;D_1:= [0,1]\times [0, K^{-1/4}].$$
The contribution stems from $D_0$ can be estimated by making use of the argument in Section 7 of \cite{BD15}. To estimate the contribution from $D_1$, we divide $D_1$ into a disjoint union of subregions
\[D_1=\bigcup \nu,\]
where $\nu$ denotes a $K^{-1/2}\times K^{-1/4}-$rectangle. Without loss of generality, we can assume that $\nu=[0, K^{-1/2}]\times [0, K^{-1/4}]$. The key point is that phase functions of the form $\psi_1(\eta_1,\eta_2)$ over the region $\nu \subset D_1$ are closed under the change of variable
\begin{equation*}
  \xi_1=\frac{\eta_1}{\lambda K^{\frac12}},\;
  \xi_2=\frac{\eta_2}{K^{\frac14}}
\end{equation*}
so that induction on scales can be used to estimate $\|\tilde{\mathcal{E}}_{\nu}\tilde{g}\|_{L^p(B_R)}$. Therefore,  we can obtain the decoupling inequality for $\Vert \tilde{\mathcal E}_{D_1}\tilde{g} \Vert_{L^p(B_R)}$.
From the discussion above, we can derive
\begin{equation}\label{equ:omegas1-1-1}
\| \mathcal{E}_{\tau} g \|_{L^p(B_R)}
\leq
 C_{\varepsilon}R^{\varepsilon}
\Big(
\sum_{\theta \subset \tau}
\| \mathcal{E}_{\theta}g \|^2_{L^p(w_{B_R})}
\Big)^{1/2},\quad 2\leq p \leq 4.
\end{equation}
This will be proved in \eqref{omegas1} below.

Combining \eqref{equ:omegall-1} and \eqref{equ:omegas1-1-1}, one can deduce that
\begin{equation}\label{equ:omeg1partcon}
   \|\mathcal E_{\Omega_1}g\|_{L^p(B_R)}\leq  C_{\varepsilon}R^{\varepsilon}\Big(\sum_{\theta\subset \Omega_1}\|\mathcal E_{\theta}g\|^2_{L^{p}(w_{B_R})}\Big)^{1/2},\;2\leq p\leq 4.
\end{equation}

By inequalities \eqref{equ:bd15}, \eqref{equ:omeest1} and \eqref{equ:omeg1partcon}, we obtain for $2\leq p\leq 4$
$$\|\mathcal E_{[0,1]^2}g\|_{L^p(B_R)}\leq \Big[C(\varepsilon)K^{O(1)}R^{\varepsilon}+2C_{\varepsilon}R^{\varepsilon}+\mathcal{D}_p\big(\tfrac{R}{K}\big)\Big]\Big(\sum_{\theta\in \mathcal{F}_3(R,4)}\|\mathcal E_{\theta}g\|^2_{L^{p}(w_{B_R})}\Big)^{1/2}.$$
This inequality implies the following recurrence inequality
\begin{equation}\label{outlinerecurrenceineq}
\mathcal{D}_p(R)\leq C(\varepsilon)K^{O(1)}R^{\varepsilon}+2C_{\varepsilon}R^{\varepsilon}
+\mathcal{D}_p\big(\tfrac{R}{K}\big).
\end{equation}
One can iterate \eqref{outlinerecurrenceineq} to deduce $\mathcal{D}_p(R)\lesssim_{\varepsilon}R^{\varepsilon}$. Therefore, we conclude the proof of Theorem \ref{thm:main}.

\vskip 0.2in




\section{Proof of Theorem \ref{thm:main}}

We are going to treat decoupling inequalities for different regions by making use of different approaches.
Firstly, we estimate the contribution from $\Omega_0$-part.

\subsection{Decoupling for $\Omega_0$}\label{subs:3.1}

In this subsection, we will establish the decoupling inequality for $\Omega_0$-part. Recall
 $$\Omega_0= \bigcup_{\lambda,\sigma} \Omega_{\lambda,\sigma}= \bigcup_{\lambda,\sigma}\Big( \bigcup_{j,m}\tau_{\lambda,\sigma}^{j,m}\Big).$$
Combining this decomposition with the Minkowski inequality and Cauchy-Schwartz inequality, we get a trivial decoupling at scale $K$ for $2\leq p \leq 4$
\begin{equation}\label{equ:trivial}
\| \mathcal{E}_{\Omega_0}g \|_{L^p(B_K)}\lesssim K^{1/2}\Big(\sum_{\tau \subset \Omega_0} \| \mathcal{E}_{\tau}g \|^2_{L^p(B_K)}\Big)^{1/2},
\end{equation}
where $\tau$ is as in \eqref{equ:taudef}.
Summing over all the balls $B_K \subset B_R$, we obtain
\begin{equation}\label{equ:trivialsumming}
\| \mathcal{E}_{\Omega_0}g \|_{L^p(B_R)}\lesssim K^{1/2}\Big(\sum_{\tau \subset \Omega_0} \| \mathcal{E}_{\tau}g \|^2_{L^p(B_R)}\Big)^{1/2}.
\end{equation}

For any given $\tau\subset\Omega_0$ of size $\lambda^{-1}K^{-1/2}\times \sigma^{-1}K^{-1/2}$, we have
\begin{lemma}\label{lem:omega0tau}
For $2\leq p \leq4$ and each $\varepsilon > 0$, there exists a positive constant $C_\varepsilon$ such that
\begin{equation}\label{equ:lamsigma}
  \|\mathcal E_{\tau}g\|_{L^p(B_R)}
  \leq C_{\varepsilon}R^{\varepsilon}\Big(\sum_{\theta \subset \tau} \| \mathcal{E}_{\theta}g \|^2_{L^p(w_{B_R})}\Big)^{1/2},
\end{equation}
where $\theta \in \mathcal{F}_3(R,4).$
\end{lemma}

With Lemma \ref{lem:omega0tau} in hand, plugging \eqref{equ:lamsigma} into \eqref{equ:trivialsumming}, we get the decoupling inequality for $\Omega_0$
\begin{equation}\label{equ:omegaoo}
\| \mathcal{E}_{\Omega_0}g \|_{L^p(B_R)}\leq C_{\varepsilon}K^{O(1)}R^{\varepsilon}\Big(\sum_{\theta \subset \Omega_0} \Vert \mathcal{E}_{\theta}g \Vert^2_{L^p(w_{B_R})}\Big)^{1/2}.
\end{equation}

\begin{proof}[{\bf Proof of Lemma \ref{lem:omega0tau}:}] Without loss of generality, we may assume that $B_R$ is centered at the origin and $\tau$ is as in \eqref{equ:tauspec}.
From \eqref{caseachangeofvariables} in Section \ref{sec:2}, we know that
\[\|\mathcal{E}_{\tau}g\|^p_{L^p(B_R)}=\lambda \sigma K^2 \|\mathcal{E}^{Parp}_{[0,1]^2}\tilde{g}\|^p_{L^p(\mathcal{L}_0(B_R))},\]
where  $\mathcal{L}_0$ denotes the map
$$(x_1,x_2,x_3)\mapsto \big(\lambda^{-1}K^{-\frac12}x_1+4\lambda^2K^{-\frac12}x_3,\sigma^{-1}K^{-\frac12}x_2+4\sigma^2K^{-\frac12}x_3,K^{-1}x_3\big),$$
and
$\mathcal{L}_0(B_R)$ denotes the image of $B_R$ with size roughly as follows
\[\lambda^{-1}K^{-1/2}R\times \sigma^{-1}K^{-1/2}R\times K^{-1}R.\]
We divide it into a finitely overlapping union of balls as follows
\[\mathcal{L}_0(B_R)=\bigcup B_{R/K}.\]
For a given $\theta \subset \tau$ such as
\[\theta = [\lambda, \lambda+\lambda^{-1}R^{-1/2}]\times[\sigma, \sigma+\sigma^{-1}R^{-1/2}],\]
and under the change of variables \eqref{caseachangeofvariables}, we deduce that the image of $\theta$ is
\[\tilde{\theta}=[0, K^{1/2}R^{-1/2}]^2.\]
We use Bourgain-Demeter's decoupling inequality \eqref{equ:bd} with $n=3$  on each $B_{R/K}$
\begin{equation}\label{omega0tauparp}
\big\|\mathcal{E}^{Parp}_{[0,1]^2}\tilde{g}\big\|_{L^p(B_{R/K})}\leq C_{\varepsilon}R^{\varepsilon}\Big(\sum_{\tilde{\theta}:K^{1/2}R^{-1/2}-square}\| \mathcal{E}^{Parp}_{\tilde{\theta}}\tilde{g}\|^2_{L^p(w_{B_{R/K}})}\Big)^{1/2}.
\end{equation}
This can be done by the argument in Section 7 of \cite{BD15}. In fact, we denote $K_p(R)$ to be the least number such that
\begin{equation}\label{equ:fkpR}
  \Vert F \Vert_{L^p(\mathbb{R}^3)}\leq K_p(R)\Big(\sum_{\bar{\theta}:R^{-1/2}-slab}\| F_{\bar{\theta}}\|^2_{L^p(\mathbb{R}^3)}\Big)^{1/2}
\end{equation}
holds for each $F$ with Fourier support in $\mathcal{N}_{R^{-1}}(S)$, where $S:= \{(\eta_1, \eta_2, \psi_0(\eta_1,\eta_2)): (\eta_1, \eta_2)\in [0,1]^2\}$ with $\psi_0(\eta_1,\eta_2)$ as in Case 1 of Section \ref{sec:2}.
It follows from \eqref{equ:fkpR} that
\begin{equation}\label{Falphadec}
\Vert F \Vert_{L^p(\mathbb{R}^3)}\leq
K_p\big(R^{2/3}\big)\Big(\sum_{\bar{\alpha}:R^{-1/3}-slab}\Vert F_{\bar{\alpha}}\Vert^2_{L^p(\mathbb{R}^3)}\Big)^{1/2}.
\end{equation}
Furthermore, from Taylor's formula, we know that on each $\bar{\alpha}$, $S$ is contained in $\mathcal{N}_{R^{-1}}(P^2)$. By invoking \eqref{equ:bd1} for this paraboloid and parabolic rescaling, we get
\begin{equation}\label{Fthetadec}
\Vert F_{\bar{\alpha}} \Vert_{L^p(\mathbb{R}^3)}\leq C_{\varepsilon}R^{\varepsilon} \Big(\sum_{\bar{\theta}:\bar{\theta}\subset \bar{\alpha}}\| F_{\bar{\theta}}\|^2_{L^p(\mathbb{R}^3)}\Big)^{1/2}.
\end{equation}
Hence, we conclude
\[K_p(R)\leq C_{\varepsilon}R^{\varepsilon}K_p\big(R^{2/3}\big),\]
which immediately leads to $K_p(R)\lesssim_{\varepsilon}R^{\varepsilon}$ by iteration. This verifies the inequality \eqref{omega0tauparp}.

Summing over all the balls $B_{R/K}\subset \mathcal{L}_0(B_R)$ on both sides of \eqref{omega0tauparp} and using Minkowski's inequality, we have
\[\|\mathcal{E}^{Parp}_{[0,1]^2}\tilde{g}\|_{L^p(\mathcal{L}_0(B_R))}\leq C_{\varepsilon}R^{\varepsilon}\Big(\sum_{\tilde{\theta}:K^{1/2}R^{-1/2}-square}\| \mathcal{E}^{Parp}_{\tilde{\theta}}\tilde{g}\|^2_{L^p(\mathcal{L}_0(B_R))}\Big)^{1/2}.\]
Taking the inverse change of variables, it follows that
\[\|\mathcal{E}_{\tau}g\|_{L^p(B_R)}\leq C_{\varepsilon}R^{\varepsilon}\Big(\sum_{\theta \subset \tau}\| \mathcal{E}_{\theta}g\|^2_{L^p(w_{B_R})}\Big)^{1/2}.\]
Therefore, we complete the proof of Lemma \ref{lem:omega0tau}.
\end{proof}

Next, we turn to discuss the contribution from $\Omega_3$-part.

\subsection{Decoupling for $\Omega_3$}\label{subsec:3.2}

By rescaling and induction on scales, we can show
\begin{lemma}\label{lem:omega3}
For $2\leq p \leq 4$, there holds
\begin{equation}\label{equ:omega3}
\|\mathcal{E}_{\Omega_3}g\|_{L^p(B_R)}\leq \mathcal{D}_p\big(\tfrac{R}{K}\big)\Big(\sum_{\theta\subset \Omega_3}\|\mathcal
{E}_{\theta}g\|^2_{L^{p}(w_{B_R})}\Big)^{1/2},
\end{equation}
where $\theta\in\mathcal{F}_3(R,4).$
\end{lemma}
\begin{proof}
 Taking the change of variables
\[\eta_i=K^{1/4}\xi_i\;(i=1,2),\]
it follows
\[\vert \mathcal{E}_{\Omega_3}g(x)\vert= \vert \mathcal{E}_{[0,1]^2}\widetilde{g}(\widetilde{x})\vert,\]
where
\[\tilde{x}:=(K^{-1/4}x_1,K^{-1/4}x_2,K^{-1}x_3),\]
and
\[\tilde{g}(\eta):= K^{-1/2}g(K^{-1/4}\eta_1,K^{-1/4}\eta_2).\]
The $\tilde{x}$-variables belong to a rectangular box $\square$ with size
\[\frac{R}{K^{1/4}}\times \frac{R}{K^{1/4}} \times \frac{R}{K},\]
 which can be written into a finitely overlapping union of balls
\[\square = \bigcup B_{R/K}.\]
By the definition of $\mathcal{D}_p(\cdot)$, we  obtain
\[\|\mathcal{E}_{[0,1]^2}\tilde{g}\|_{L^p(B_{R/K})}\leq \mathcal{D}_p\big(\tfrac{R}K\big)\Big(\sum_{\tilde{\theta} \in \mathcal{F}_3(\frac{R}{K},4)} \| \mathcal{E}_{\tilde{\theta}}\tilde{g}\|^2_{L^p(w_{B_{R/K}})}\Big)^{1/2}.\]
Summing over all the cubes $B_{R/K}\subset \square$ and using Minkowski's inequality, we have
\[\|\mathcal{E}_{[0,1]^2}\tilde{g}\|_{L^p(\square)}\leq \mathcal{D}_p\big(\tfrac{R}K\big)\Big(\sum_{\tilde{\theta} \in \mathcal{F}_3(\frac{R}{K},4)}\| \mathcal{E}_{\tilde{\theta}}\tilde{g}\|^2_{L^p(\square)}\Big)^{1/2}.\]
Taking the inverse change of variables, one has
\[\| \mathcal{E}_{\Omega_3}g \|_{L^p(B_R)}\leq \mathcal{D}_p\big(\tfrac{R}K\big)\Big(\sum_{\theta \subset \Omega_3} \| \mathcal{E}_{\theta}g \|^2_{L^p(w_{B_R})}\Big)^{1/2}.\]
Thus, we complete the proof of Lemma \ref{lem:omega3}.
\end{proof}

Finally, we deal with the contribution from $\Omega_1$-part.

\subsection{Decoupling for $\Omega_1$}\label{subs:decoup1}
Recall the decomposition
\begin{align*}
  \Omega_1= & \bigcup_\lambda \Omega_{1,\lambda},\quad
  \Omega_{1,\lambda} = \bigcup_\lambda \bigcup_j\tau_{\lambda}^j=:\bigcup \tau_{\lambda},
\end{align*}
where we abbreviate $\tau_{\lambda}^{j}$ by $\tau_{\lambda}$.
First, by reduction of dimension arguments, we are able to prove the following result.
\begin{lemma}\label{lem:3.3}
For $2\leq p\leq 6$ and any $\varepsilon>0$, there exists a constant $C_\varepsilon$ such that
\begin{equation}\label{omegall}
\| \mathcal{E}_{\Omega_1}g \|_{L^p(B_R)}\leq C_{\varepsilon}K^{\varepsilon}\Big(\sum_{\lambda}\sum_{\tau_{\lambda} \subset \Omega_{1,\lambda}} \| \mathcal{E}_{\tau_{\lambda}}g \|^2_{L^p(w_{B_R})}\Big)^{1/2}.
\end{equation}
\end{lemma}

To prove Lemma \ref{lem:3.3}, we employ the following lemma from \cite{Yang21}.
\begin{lemma}\label{lem:gamwideg}
Let $\Gamma_{\lambda}=\{(t,t^4): t\in [\lambda,2\lambda]\}$ and $2\leq p\leq 6$. For any $\varepsilon>0$, there holds
\begin{equation}\label{equ:gamgthet}
  \|G\|_{L^p(\mathbb{R}^2)}\leq C_{\varepsilon}K^{\varepsilon} \Big(\sum_{\tau}\Vert G_{\tau}\Vert^2_{L^p(\mathbb{R}^2)}\Big)^{1/2},\;G_{\tau}:=\mathcal F^{-1}(\widehat{G}\chi_{\tau}),
\end{equation}
where ${\rm supp}\;\hat{G}\subset \mathcal{N}_{K^{-1}}(\Gamma_{\lambda})$ and $\tau$'s are $\lambda^{-1}K^{-\frac12}\times K^{-1}$-rectangles.
\end{lemma}

With Lemma \ref{lem:gamwideg} in hand, we prove Lemma \ref{lem:3.3} by freezing the $x_2$ variable as follows. Fix a bump function $\varphi \in C^{\infty}_c(\mathbb{R}^3)$ with ${\rm supp}\; \varphi \subset B(0,1)$ and $\vert \check{\varphi}(x)\vert \geq 1$ for all $x \in B(0,1)$. Define $F:= \check{\varphi}_{K^{-1}}\mathcal{E}_{\Omega_1}g$, where $\varphi_{K^{-1}}(\xi):= K^3 \varphi(K\xi),\; \xi \in \mathbb{R}^3$.
We denote $F(\cdot,x_2,\cdot)$ by $G$. From Guth \cite{Guth18}, it is easy to see that ${\rm supp}\;\hat{G}$ is contained in the projection of ${\rm supp}\;\hat{F}$ on the plane $\xi_2=0$, that is, in the $K^{-1}$-neighborhood of $\Gamma_{\lambda}$. We employ Lemma \ref{lem:gamwideg} to get
\[ \|G\|_{L^p(\mathbb{R}^2)}\leq C_{\varepsilon}K^{\varepsilon} \Big(\sum_{\tau}\Vert G_{\tau}\Vert^2_{L^p(\mathbb{R}^2)}\Big)^{1/2},\]
i.e.,
\[\|F(\cdot,x_2,\cdot)\|_{L^p(\mathbb{R}^2)}\leq C_{\varepsilon}K^{\varepsilon} \Big(\sum_{\tau}\Vert F_{\tau}(\cdot,x_2,\cdot)\Vert^2_{L^p(\mathbb{R}^2)}\Big)^{1/2},\]
where $F_{\tau}(x):=\check{\varphi}_{K^{-1}}\mathcal{E}_{\tau}g$.
Integrating on both sides of the above inequality with respect to $x_2$-variable from $-\infty$ to $\infty$, we derive
\[\|F\|_{L^p(\mathbb{R}^3)}\leq C_{\varepsilon}K^{\varepsilon} \Big(\sum_{\tau}\Vert F_{\tau}\Vert^2_{L^p(\mathbb{R}^3)}\Big)^{1/2}.\]
Thus, we have
\begin{align*}
  \| \mathcal{E}_{\Omega_{1,\lambda}}g \|_{L^p(B_K)}\lesssim& \|F\|_{L^p(\mathbb{R}^3)}
\leq C_{\varepsilon}K^{\varepsilon} \Big(\sum_{\tau}\Vert F_{\tau}\Vert^2_{L^p(\mathbb{R}^3)}\Big)^{1/2}\\
\leq& C_{\varepsilon}K^{\varepsilon}\Big(\sum_{\tau_{\lambda} \subset \Omega_{1,\lambda}} \| \mathcal{E}_{\tau}g \|^2_{L^p(w_{B_K})}\Big)^{1/2}.
\end{align*}
It follows that
\[\| \mathcal{E}_{\Omega_1}g \|_{L^p(B_K)}\leq C_{\varepsilon}K^{\varepsilon}\Big(\sum_{\lambda}\sum_{\tau_{\lambda} \subset \Omega_{1,\lambda}} \| \mathcal{E}_{\tau_{\lambda}}g \|^2_{L^p(w_{B_K})}\Big)^{1/2}.\]
Summing over all the balls $B_K\subset B_R$, we get the inequality \eqref{omegall}, as required.
\vskip 0.2in
Next, we estimate the $\| \mathcal{E}_{\tau_{\lambda}}g \|_{L^p(B_R)}$ on the right hand side of \eqref{omegall}.
For this purpose, we consider the following surfaces:
\begin{equation}\label{equ:f224def}
F^2_{2,4}:=\{(\xi_1,\xi_2,\phi_1(\xi_1)+\xi_2^4):(\xi_1,\xi_2)\in [0,1]^2\},
\end{equation}
where $\phi_1(\xi_1)$ satisfies $\phi_1''\sim 1$, $\vert \phi_1^{(3)} \vert \lesssim 1, \vert \phi_1^{(4)} \vert \lesssim 1$ and  $\phi_1^{(\ell)} =0,\; \ell \geq 5$ on $[0,1]$.
Now, we construct the decomposition of $[0,1]^2$ associated with the surfaces $F^2_{2,4}$. We first divide $[0,1]$ into
\begin{equation*}
  [0,1]=\bigcup_k I_k,
\end{equation*}
where $I_0=[0,R^{-\frac14}]$ and
$$I_k=[2^{k-1}R^{-\frac14},2^{k}R^{-\frac14}], \;\;\text{for}\;\; 1\leq k \leq \big[\tfrac{1}{4}\log_{2}R\big].$$
Furthermore, for each $k\geq 1$, we divide $I_k$ into
\[I_k=\bigcup_{\mu=1}^{2^{2(k-1)}} I_{k,\mu},\]
with
$$I_{k,\mu}=\big[2^{k-1}{R^{-\frac14}}+(\mu-1)2^{-(k-1)}R^{-\frac14},2^{k-1}R^{-\frac14}+ \mu 2^{-(k-1)}R^{-\frac14}\big].$$
Then, we have the following decomposition
\begin{equation}\label{equ:01theta}
  [0,1]^{2}=\bigcup \tilde{\theta},
\end{equation}
where $\tilde{\theta}\in \mathcal{F}_3(R,2,4)$ and
\begin{align*}
   \mathcal{F}_3(R,2,4):=\Big\{&[a, a+R^{-1/2}]\times I_{k,\mu},\;[a, a+R^{-1/2}]\times I_0:\\
   & a \in [0, 1-R^{-1/2}]\cap R^{-1/2}\mathbb{Z},
   \quad 1\leq k \leq \big[\tfrac{1}{4}\log_{2}R\big],\;1\leq \mu \leq 2^{2(k-1)}\Big\}.
\end{align*}

We establish the decoupling inequality  for the surfaces $F_{2,4}^2$ associated with the decomposition \eqref{equ:01theta} as follows.

\begin{lemma}\label{lem:24}
For $2\leq p \leq 4$ and any $\varepsilon>0$, there exists a constant $C_\varepsilon$ such that
\begin{equation}\label{surfacenew}
 \| \tilde{\mathcal{E}}_{[0,1]^2}f \|_{L^p(B_R)}\leq C_{\varepsilon}R^{\varepsilon}\Big(\sum_{\tilde{\theta} \in \mathcal{F}_3(R,2,4)} \|\tilde{\mathcal{E}} _{\tilde{\theta}}f \|^2_{L^p(w_{B_R})}\Big)^{1/2},
\end{equation}
where $\tilde{\mathcal{E}}_{[0,1]^2}$ denotes the Fourier extension operator associated with the surfaces $F^2_{2,4}$.
\end{lemma}
Assume that Lemma \ref{lem:24} holds for a while, for each $\lambda$ and any $\tau_{\lambda} \subset \Omega_{1,\lambda}$, we claim that
\begin{equation}\label{omegas1}
\| \mathcal{E}_{\tau_{\lambda}} g \|_{L^p(B_R)}
\leq
 C_{\varepsilon}\big(\tfrac{R}{K}\big)^{\varepsilon}
\Big(
\sum_{\theta \subset \tau_{\lambda}}
\| \mathcal{E}_{\theta}g \|^2_{L^p(w_{B_R})}
\Big)^{1/2},\quad 2\leq p \leq 4,
\end{equation}
where $\theta\in\mathcal{F}_3(R,4).$ Without loss of generality, we may
assume that $\tau_{\lambda} = [\lambda, \lambda + \lambda^{-1}K^{-1/2}]\times [0, K^{-1/4}]$.
Taking the change of variables
\begin{equation}\label{omega1tau}
\xi_1= \lambda+ \lambda^{-1}K^{-1/2}\eta_1,\; \xi_2= K^{-1/4}\eta_2,
\end{equation}
we get
\[\|\mathcal{E}_{\tau}g\|^p_{L^p(B_R)}=\lambda K^{7/4} \|\tilde{\mathcal{E}}_{[0,1]^2}\tilde{g}\|^p_{L^p(\mathcal{L}_{\lambda}(B_R))},\]
where
\[\tilde{g}(\eta_1,\eta_2)= \lambda^{-1}K^{-3/4}g(\lambda+ \lambda^{-1}K^{-1/2}\eta_1, K^{-1/4}\eta_2).\]
Here $\mathcal{L}_{\lambda}$ denotes the map
$$(x_1,x_2,x_3)\mapsto\big(\lambda^{-1}K^{-\frac12}x_1+4\lambda^2K^{-\frac12}x_3,K^{-\frac14}x_2,K^{-1}x_3\big),$$
and
$\mathcal{L}_{\lambda}(B_R)$ denotes the image of $B_R$ with size roughly
\[\lambda^{-1/2}K^{-1/2}R \times K^{-1/4}R\times K^{-1}R.\]
We rewrite $\mathcal{L}_{\lambda}(B_R)$ into a finitely overlapping union of balls as follows
\[\mathcal{L}_{\lambda}(B_R)=\bigcup B_{R/K}.\]
Now, we employ Lemma \ref{lem:24} on each $B_{R/K}$ to estimate the
\[\|\tilde{\mathcal{E}}_{[0,1]^2}\tilde{g}\|_{L^p(B_{R/K})}.\]
To do this, we need verify that for $\theta \subset \tau_{\lambda}$, its image $\tilde{\theta}$ belongs to $\mathcal{F}_3(\frac{R}{K},2,4)$ under the change of variables \eqref{omega1tau}. We only need to consider the following two  cases:

{\bf Case(a): $\theta\subset [\lambda,2\lambda]\times[0, R^{-1/4}]$.} In this case, we may assume that $$\theta = [\lambda, \lambda+\lambda^{-1}R^{-1/2}]\times[0, R^{-1/4}].$$
Under the change of variables \eqref{omega1tau}, we see that
\[\tilde{\theta}=[0, K^{1/2}R^{-1/2}]\times [0, K^{1/4}R^{-1/4}]\in \mathcal{F}_3\big(\tfrac{R}{K},2,4\big).\]

{\bf Case(b): $\theta \subset [\lambda,2\lambda]\times[\sigma, 2\sigma].$} In this case, we may assume that  $$\theta = [\lambda, \lambda+\lambda^{-1}R^{-1/2}]\times[\sigma, \sigma+\sigma^{-1}R^{-1/2}],$$ where $\sigma$ is a dyadic number with $R^{-1/4}\leq \sigma \leq \frac{1}{2}$.
Under the change of variables \eqref{omega1tau}, one has
\[\tilde{\theta}=\big[0, \big(\tfrac{R}{K}\big)^{-1/2}\big]\times \big[\tilde{\sigma}, \tilde{\sigma}+\tilde{\sigma}^{-1}\big(\tfrac{R}{K}\big)^{-1/2}\big]\in \mathcal{F}_3\big(\tfrac{R}{K},2,4\big),\]
where $\tilde{\sigma}:= K^{1/4}\sigma$ is also a dyadic number.

By Lemma \ref{lem:24}, we obtain
\[\|\tilde{\mathcal{E}}_{[0,1]^2}\tilde{g}\|_{L^p(B_{R/K})}\leq C_{\varepsilon}\big(\tfrac{R}{K}\big)^{\varepsilon}\Big(\sum_{\tilde\theta\in \mathcal{F}_3(\frac{R}{K},2,4)}\| \tilde{\mathcal{E}}_{\tilde\theta}\tilde{g}\|^2_{L^p(w_{B_{R/K}})}\Big)^{1/2}.\]
Summing over all the balls $B_{R/K}\subset \mathcal{L}_{\lambda}(B_R)$ and using Minkowski's inequality, we have
\[\|\tilde{\mathcal{E}}_{[0,1]^2}\tilde{g}\|_{L^p(\mathcal{L}_{\lambda}(B_R))}\leq C_{\varepsilon}\big(\tfrac{R}{K}\big)^{\varepsilon}\Big(\sum_{\tilde\theta\in \mathcal{F}_3(\frac{R}{K},2,4)}\| \tilde{\mathcal{E}}_{\tilde\theta}\tilde{g}\|^2_{L^p(\mathcal{L}_{\lambda}(B_R))}\Big)^{1/2}.\]
Taking the inverse change of variables, we deduce that
\[\|\mathcal{E}_{\tau_{\lambda}}g\|_{L^p(B_R)}\leq C_{\varepsilon}\big(\tfrac{R}{K}\big)^{\varepsilon}\Big(\sum_{\theta \subset \tau}\| \mathcal{E}_{\theta}g\|^2_{L^p(w_{B_R})}\Big)^{1/2}.\]
Thus, we prove the claim \eqref{omegas1}.

Plugging \eqref{omegas1} into \eqref{omegall}, one has
\begin{align}\nonumber
  \| \mathcal{E}_{\Omega_1}g \|_{L^p(B_R)}\leq& C_{\varepsilon}K^{\varepsilon}\big(\tfrac{R}{K}\big)^{\varepsilon}\Big(\sum_{\lambda}\sum_{\tau_{\lambda}}\sum_{\theta \subset \tau_{\lambda}} \| \mathcal{E}_{\theta}g \|^2_{L^p(w_{B_R})}\Big)^{1/2}\\\label{equ:controme1est}
  =& C_{\varepsilon}R^{\varepsilon}\Big(\sum_{\theta \subset \Omega_1} \| \mathcal{E}_{\theta}g \|^2_{L^p(w_{B_R})}\Big)^{1/2},
\end{align}
where $\theta \in \mathcal{F}_3(R,4)$.
Therefore, we obtain the contribution from $\Omega_1$-part under the assumption of  Lemma \ref{lem:24}.

Now, we return to prove Lemma \ref{lem:24}.

\begin{proof}[{\bf Proof of Lemma \ref{lem:24}:}]
Denote the least constant such that \eqref{surfacenew} holds by $\tilde{\mathcal{D}}_p(R)$. Divide $[0,1]^2$ into $[0,1]^2=D_0 \bigcup D_1$, where
$$D_0:= [0,1]\times [K^{-1/4},1],\;D_1:= [0,1]\times [0, K^{-1/4}].$$
It is easy to see that
\begin{equation}\label{equ:eaomeg01}
  \| \tilde{\mathcal{E}}_{[0,1]^2}f \|_{L^p(B_R)}\leq \| \tilde{\mathcal{E}}_{D_0}f \|_{L^p(B_R)}+\| \tilde{\mathcal{E}}_{D_1}f \|_{L^p(B_R)}.
\end{equation}

{\bf Step 1: The contribution stems from $D_0$.}
We have a trivial decoupling at scale $K$ for $p\geq 2$
\begin{equation}\label{equ:trivial-1v}
\| \tilde{\mathcal{E}}_{D_0}f \|_{L^p(B_K)}\lesssim K^{O(1)}\Big(\sum_{\nu \subset D_0} \| \tilde{\mathcal{E}}_{\nu}f \|^2_{L^p(B_K)}\Big)^{1/2},
\end{equation}
where $\nu\in\mathcal{F}_3(K,2,4).$
Summing over all the balls $B_K \subset B_R$, we obtain
\begin{equation}\label{equ:trivialsummingnew}
\| \tilde{\mathcal{E}}_{D_0}f \|_{L^p(B_R)}\leq CK^{O(1)}\Big(\sum_{\nu \subset D_0} \| \tilde{\mathcal{E}}_{\nu}f \|^2_{L^p(B_R)}\Big)^{1/2},\;\nu \in\mathcal{F}_3(K,2,4).
\end{equation}
For any given $\nu \subset D_0$ of size $K^{-1/2}\times \sigma^{-1}K^{-1/2}$, where $\sigma$ is a dyadic number with $K^{-\frac14}\leq \sigma \leq \frac{1}{2}$, we claim that
\begin{equation}\label{equ:lamsigmanew}
  \|\tilde{\mathcal {E}}_{\nu}f\|_{L^p(B_R)}
  \leq C_{\varepsilon}\big(\tfrac{R}{K}\big)^{\varepsilon}\Big(\sum_{\vartheta \subset \nu} \| \tilde{\mathcal{E}}_{\vartheta}f \|^2_{L^p(w_{B_R})}\Big)^{1/2},\;\; 2\leq p \leq 4,
\end{equation}
where $\vartheta \in\mathcal{F}_3(R,2,4)$. Plugging \eqref{equ:lamsigmanew} into \eqref{equ:trivialsummingnew}, it follows that
\begin{equation}\label{omegans}
\| \tilde{\mathcal{E}}_{D_0}f \|_{L^p(B_R)}\leq C_{\varepsilon}K^{O(1)}R^{\varepsilon}\Big(\sum_{\vartheta \subset D_0} \Vert \tilde{\mathcal{E}}_{\vartheta}f \Vert^2_{L^p(w_{B_R})}\Big)^{1/2},\;\;2\leq p \leq 4.
\end{equation}
Now, we turn to prove the claim \eqref{equ:lamsigmanew}.
Without loss of generality, we may assume that $\nu = [0, K^{-1/2}]\times [\sigma, \sigma+\sigma^{-1}K^{-1/2}]$.
Taking the change of variables
\begin{equation}\label{lemma35omega1changeofvariables}
\xi_1= K^{-1/2}\eta_1,\; \xi_2= \sigma+ \sigma^{-1}K^{-1/2}\eta_2,
\end{equation}
we get \[\|\tilde{\mathcal{E}}_{\nu}f\|^p_{L^p(B_R)}=\sigma K^2 \big\|\mathcal{E}^{Parp}_{[0,1]^2}\tilde{f}\big\|^p_{L^p(\mathcal{T}_0(B_R))},\]
where $\tilde{f}(\eta_1,\eta_2):= \sigma^{-1}K^{-1}f(K^{-1/2}\eta_1, \sigma+ \sigma^{-1}K^{-1/2}\eta_2)$,  $\mathcal{T}_0$ denotes the map
\[\mathcal{T}_0:\;(x_1,x_2,x_3)\rightarrow \big(K^{-\frac12}x_1,\sigma^{-1}K^{-\frac12}x_2+4\sigma^2K^{-\frac12}x_3K^{-1}x_3\big),\]
and $\mathcal{T}_0(B_R)$ denotes the image of $B_R$ with size roughly
\[K^{-1/2}R \times \sigma^{-1}K^{-1/2}R\times K^{-1}R.\]
We rewrite it into a finitely overlapping union of balls as following
\[\mathcal{T}_0(B_R)=\bigcup B_{R/K}.\]
Using a similar argument as in the proof of \eqref{omega0tauparp}, we get
\[\big\|\mathcal{E}^{Parp}_{[0,1]^2}\tilde{f}\big\|_{L^p(B_{R/K})}\leq C_{\varepsilon}\big(\tfrac{R}{K}\big)^{\varepsilon}\Big(\sum_{\vartheta:K^{1/2}R^{-1/2}-square}\| \mathcal{E}^{Parp}_{\vartheta}\tilde{f}\|^2_{L^p(w_{B_{R/K}})}\Big)^{1/2}.\]
Summing over all the balls $B_{R/K}\subset \mathcal{T}_0(B_R)$ and using Minkowski's inequality, we have
\[\|\mathcal{E}^{Parp}_{[0,1]^2}\tilde{f}\|_{L^p(\mathcal{T}_0(B_R))}\leq C_{\varepsilon}\big(\tfrac{R}{K}\big)^{\varepsilon}\Big(\sum_{\vartheta:K^{1/2}R^{-1/2}-square}\| \mathcal{E}^{Parp}_{\vartheta}\tilde{f}\|^2_{L^p(\mathcal{T}_0(B_R))}\Big)^{1/2}.\]
Taking the inverse change of variables, we obtain
\[\|\tilde{\mathcal{E}}_{\nu}f\|_{L^p(B_R)}\leq C_{\varepsilon}\big(\tfrac{R}{K}\big)^{\varepsilon}\Big(\sum_{\vartheta \subset \nu}\| \tilde{\mathcal{E}}_{\vartheta}f\|^2_{L^p(w_{B_R})}\Big)^{1/2}.\]
Hence, we conclude the claim  \eqref{equ:lamsigmanew}.

{\bf Step 2: The contribution stems from $D_1$.} We first prove
\begin{equation}\label{omegac}
\| \tilde{\mathcal{E}}_{D_1}f \|_{L^p(B_R)}\leq C_{\varepsilon}K^{\varepsilon}\Big(\sum_{\nu \subset D_1} \| \tilde{\mathcal{E}}_{\nu}f \|^2_{L^p(w_{B_R})}\Big)^{1/2},\;\; 2\leq p \leq 6,
\end{equation}
where each $\nu$ denotes a $K^{-1/2}\times K^{-1/4}$-rectangle contained in the region $D_1$.
By the same argument as in the proof of Lemma \ref{lem:3.3}, it suffices to prove
\begin{lemma}\label{lem:curve1new}
Let $\Gamma_{\phi_1}:=\{(t,\phi_1(t)): t\in [0,1]\}$ satisfy
\[\phi_1''\sim 1,\; \vert\phi_1^{(3)}\vert \lesssim 1,\; \vert\phi_1^{(4)}\vert \lesssim 1 \;\; \text{and} \;\; \phi^{(\ell)}_1 =0,\; \ell \geq 5 \; \text{on} \;[0,1].\]
For any $\varepsilon>0$, there exists a constant $C_\varepsilon$ such that
\begin{equation}\label{equ:curve1new}
\|G\|_{L^p(\mathbb{R}^2)}\leq C_{\varepsilon}K^{\varepsilon} \Big(\sum_{\nu}\Vert G_{\nu}\Vert^2_{L^p(\mathbb{R}^2)}\Big)^{1/2},\;2\leq p\leq 6,
\end{equation}
where ${\rm supp}\;\hat{G}\subset \mathcal{N}_{K^{-1}}(\Gamma_{\phi_1})$, $\nu$'s are $K^{-\frac12}\times K^{-1}$-rectangles and $G_{\nu}:=\mathcal F^{-1}(\widehat{G}\chi_{\nu})$.
\end{lemma}
By the modification as in the verification of \eqref{omega0tauparp}, Lemma \ref{lem:curve1new} follows by Bourgain-Demeter's decoupling inequality \eqref{equ:bd} for $n=2$.

Now, we turn to estimate each term $ \| \tilde{\mathcal{E}}_{\nu}f \|_{L^p(w_{B_R})}$ in the right hand side of \eqref{omegac}.
For each $\nu\subset D_1$, we claim that
\begin{equation}\label{omegas}
\| \tilde{\mathcal{E}}_{\nu}f \|_{L^p(B_R)}
\leq
 \tilde{\mathcal{D}}_p\big(\tfrac{R}{K}\big)
\Big(
\sum_{\vartheta \subset \nu}
\| \tilde{\mathcal{E}}_{\vartheta}f \|^2_{L^p(w_{B_R})}
\Big)^{1/2},\quad 2\leq p \leq 4,
\end{equation}
where $\vartheta \in\mathcal{F}_3(R,2,4).$
Plugging \eqref{omegas} into \eqref{omegac}, we obtain
\begin{equation}\label{equ:tildome1gans}
\| \tilde{\mathcal{E}}_{D_1}f \|_{L^p(B_R)}\leq C_{\varepsilon}K^{\varepsilon}\tilde{\mathcal{D}}_p\big(\tfrac{R}{K}\big)\Big(\sum_{\vartheta \subset D_1} \Vert \tilde{\mathcal{E}}_{\vartheta}f \Vert^2_{L^p(w_{B_R})}\Big)^{1/2},\;\; 2\leq p \leq 4.
\end{equation}

Next, we turn to show the claim \eqref{omegas}.
Without loss of generality, we may assume that $\nu = [0, K^{-1/2}]\times [0, K^{-1/4}]$.
Taking the change of variables $$\xi_1= K^{-1/2}\eta_1,\; \xi_2= K^{-1/4}\eta_2,$$ we get
\[\|\tilde{\mathcal{E}}_{\nu}f\|^p_{L^p(B_R)}=\sigma K^2 \|\tilde{\mathcal{E}}_{[0,1]^2}\tilde{f}\|^p_{L^p(\mathcal{T}_1(B_R))},\]
where $\tilde{f}(\eta_1,\eta_2):= K^{-3/4}f(K^{-1/2}\eta_1, K^{-1/4}\eta_2)$. Here $\mathcal{T}_1$ denotes the map
\[(x_1,x_2,x_3)\rightarrow (K^{-1/2}x_1,K^{-1/4}x_2,K^{-1}x_3)\]
and $\mathcal{T}_1(B_R)$ denotes the image of $B_R$ with size roughly as follows
\[K^{-1/2}R \times K^{-1/4}R\times K^{-1}R.\]
Using the similar argument as in Case(a) and Case(b) in the proof of \eqref{omegas1}, one can show that the image of $\vartheta$ belongs to the decomposition $\mathcal{F}_3(\frac{R}{K},2,4)$. We write $\mathcal{T}_1(B_R)$ into a finitely overlapping union of balls as follows
\[\mathcal{T}_1(B_R)=\bigcup B_{R/K}.\]
By the definition of $\tilde{\mathcal{D}}_p(\cdot),$ we obtain
\[\|\tilde{\mathcal{E}}_{[0,1]^2}\tilde{f}\|_{L^p(B_{R/K})}\leq \tilde{\mathcal{D}}_p\big(\tfrac{R}{K}\big)\Big(\sum_{\vartheta\in \mathcal{F}_3(\frac{R}{K},2,4)}\| \tilde{\mathcal{E}}_{\vartheta}\tilde{f}\|^2_{L^p(w_{B_{R/K}})}\Big)^{1/2}.\]
Summing over all the balls $B_{R/K}\subset \mathcal{T}_1(B_R)$ and using Minkowski's inequality, we have
\[\|\tilde{\mathcal{E}}_{[0,1]^2}\tilde{f}\|_{L^p(\mathcal{T}_1(B_R))}\leq C_{\varepsilon}R^{\varepsilon}\Big(\sum_{\vartheta\in \mathcal{F}_3(\frac{R}{K},2,4)}\| \tilde{\mathcal{E}}_{\vartheta}\tilde{f}\|^2_{L^p(\mathcal{T}_1(B_R))}\Big)^{1/2}.\]
Taking the inverse change of variables, we obtain
\[\|\tilde{\mathcal{E}}_{\nu}f\|_{L^p(B_R)}\leq \tilde{\mathcal{D}}_p\big(\tfrac{R}{K}\big)\Big(\sum_{\vartheta \subset \nu}\| \tilde{\mathcal{E}}_{\vartheta}f\|^2_{L^p(w_{B_R})}\Big)^{1/2}.\]
Hence, we obtain the claim \eqref{omegas}.

\vskip 0.1in

Using \eqref{equ:eaomeg01}, \eqref{omegans} and \eqref{equ:tildome1gans}, we derive
$$ \| \tilde{\mathcal{E}}_{[0,1]^2}f \|_{L^p(B_R)}\leq \Big(C(\varepsilon)K^{O(1)}R^{\varepsilon}+C_{\varepsilon}K^{\varepsilon}\tilde{\mathcal{D}}_p\big(\tfrac{R}{K}\big)\Big)
\Big(\sum_{\vartheta \subset \mathcal{F}_3(R,2,4)} \Vert \tilde{\mathcal{E}}_{\vartheta}f \Vert^2_{L^p(w_{B_R})}\Big)^{1/2}.$$
This inequality together with the definition of $\tilde{\mathcal{D}}_p(R)$ yields
\[\tilde{\mathcal{D}}_p(R)\leq C(\varepsilon)K^{O(1)}R^{\varepsilon}+C_{\varepsilon}K^{\varepsilon}\tilde{\mathcal{D}}_p\big(\tfrac{R}{K}\big).\]
Iterating the above inequality $m = [\log_K R]$ times, we derive that $\tilde{\mathcal{D}}_p(R)\lesssim_{\varepsilon}R^{\varepsilon}$.

Therefore, we complete the proof of Lemma \ref{lem:24}.
\end{proof}

%
%
%

\begin{remark}
The argument in this paper also implies the following result. Let $2\leq p\leq4$ and $m \geq 4$ be an even number. For any $\varepsilon>0$ , there holds
\begin{equation*}
  \|\mathcal{E}_{[0,1]^2}g\|_{L^p(B_R)}\leq C(\varepsilon,p)R^{\varepsilon}\Big(\sum_{\theta\in \mathcal{F}_3(R,m)}\|\mathcal E_{\theta}g\|^2_{L^{p}(w_{B_R})}\Big)^{1/2},
\end{equation*}
where
\[\mathcal E_{Q}g(x)=\int_{Q}g(\xi_1,\xi_2)e(x_1\xi_1+x_2\xi_2+x_3(\xi_1^m+\xi_2^m))\;d\xi_1\;d\xi_2.\]
\end{remark}




\begin{center}

\end{center}

\end{document}